\documentclass[11pt]{amsart}
\usepackage{graphicx}
\usepackage{amssymb}
\usepackage{amsmath}
\usepackage{amsthm,amsfonts,bbm}
\usepackage{amscd}
\usepackage{geometry}
\usepackage[all,2cell]{xy}
\usepackage{epsfig,epstopdf}
\usepackage[backref]{hyperref}
\usepackage{tikz}
\usepackage{caption}
\usepackage{subcaption}
\UseAllTwocells \SilentMatrices
\newtheorem{thm}{Theorem}[section]
\newtheorem{cor}[thm]{Corollary}
\newtheorem{lem}[thm]{Lemma}

\theoremstyle{definition}
\newtheorem{defi}[thm]{Definition}
\theoremstyle{remark}
\newtheorem{rmk}[thm]{\bf Remark}

\numberwithin{equation}{section}
\numberwithin{figure}{section}
\def \cl{\text{Cl}}
\def\deg{\text{deg}}
\def \d{\delta}

\def \deg{\text{deg}}
\def \dim{\text{dim}}

\def \im{\text{im~}}

\def \K{\mathcal{K}}
\def \L{\mathcal{L}}
\def \lk{\text{Lk}}
\def\la{\lambda}
\def \lamax{\lambda_{\max}}

\def \p{\partial}
\def \Q{\mathcal{Q}}
\def \R{\mathbb{R}}
\def \s{\sigma}
\def \st{\text{St}}
\def \spec{\text{~Spec}}

\DeclareMathOperator{\sgn}{sgn}

\begin{document}
\title[Largest eigenvalue and balancedness of complex]{The largest Laplacian eigenvalue and the balancedness of simplicial complexes}

\author[Y.-Z. Fan]{Yi-Zheng Fan*}
\address{Center for Pure Mathematics, School of Mathematical Sciences, Anhui University, Hefei 230601, P. R. China}
\email{fanyz@ahu.edu.cn}
\thanks{*The corresponding author.
This work was supported by National Natural Science Foundation of China (Grant Nos. 12331012, 12171002).}

\author[H.-F. Wu]{Hui-Feng Wu}
\address{Center for Pure Mathematics, School of Mathematical Sciences, Anhui University, Hefei 230601, P. R. China}
\email{wuhf@stu.ahu.edu.cn}

\author[Y. Wang]{Yi Wang}
\address{Center for Pure Mathematics, School of Mathematical Sciences, Anhui University, Hefei 230601, P. R. China}
\email{wangy@ahu.edu.cn}



\begin{abstract}
Let $K$ be a simplical complex, and let $\L_i^{up}(K), \Q_i^{up}(K)$ be the $i$-th up Laplacian and signless Laplacian of $K$, respectively.
In this paper we proved that the largest eigenvalue of $\L_i^{up}(K)$ is not greater than the largest eigenvalue of $\Q_i^{up}(K)$; furthermore, if $K$ is $(i+1)$-path connected, then the equality holds if and only if the $i$-th incidence signed graph $B_i(K)$ of $K$ is balanced.
As an application we provided an upper bound for the largest eigenvalue of the $i$-th up Laplacian of  $K$, which improves the bound given by Horak and Jost and generalizes the result of Anderson and Morley on graphs.
We characterized the balancedness of simplicial complexes under operations such as wedge sum, join, Cartesian product and duplication of motifs.
For each $i \ge 0$, by using wedge sum or duplication of motifs, we can construct an infinitely many $(i+1)$-path connected simplicial complexes $K$ with $B_i(K)$ being balanced.
\end{abstract}

\maketitle

\section{Introduction}
The research of the Laplacian operator on graphs has a long and rich history. In
the mid-19th century Kirchhoff \cite{Kirch} studied the electric networks and formulated the celebrated matrix-tree theorem.
In the early 1970s, Fiedler \cite{Fied} established a relationship between the  second smallest Laplacian eigenvalue and the connectivity of a graph.
The normalized graph Laplacian was introduced by Bottema \cite{Bot} who studied a transition probability on graphs; also see \cite{Chung93} for more details.
 Eckmann \cite{Eckmann} extended the Laplace operator from graphs to simplicial complexes, providing  the discrete version of the Hodge theorem, which can be expressed as
$$ \text{ker}(\d_i^* \d_i+ \d_{i-1} \d_{i-1}^*) = \tilde{H}^i(K,\mathbb{R}),$$
where $\d_i^* \d_i+ \d_{i-1} \d_{i-1}^*$ is the higher order combinatorial Laplacian of the simplicial complex $K$.

There have been many studies on the Laplacian eigenvalues of simplicial complexes.
Duval and Reiner \cite{Duv} show that the combinatorial Laplace operator of a shifted simplicial complex has an integral spectrum.
 By using combinatorial Laplace operators, Dong and Wachs \cite{Dong} gave an elegant proof of Bouc's result \cite{Bouc} on the decomposition of the representation of the symmetric group on the homology of a matching complex, and proved the spectrum of the Laplace operator of the matching complex is also integral.
   Some remarkable recent results on combinatorial Laplacian for simplicial complexes involve lots of algebraic and geometric aspects for complexes, which are applied in the study of independent number, chromatic number, theta number \cite{Bachoc, Golubev,HJ13B}.

Horak and Jost \cite{HJ13B} developed a general framework for Laplace operators defined in terms of the combinatorial structure of a simplicial complex, including the combinatorial Laplacian and the normalized Laplacian.
In their paper they provided some upper bounds for the largest Laplacian eigenvalues in terms of degrees and weights of faces.
Rotman \cite{Rot} introduced the covering complex, which was adopted by Gustavson \cite{Gust} to study the Laplacian spectrum.
Horak and Jost \cite{HJ13B} proved that the Laplacian spectrum of a covering complex contains the Laplacian spectrum of its underlying complex.
Fan and Song \cite{Fan} show that the spectrum of a $2$-fold covering simplicial complex is the union of the spectrum of the underlying simplicial complex and the spectrum of an incidence-signed simplicial complex, which generalizes Bilu and Linial's result on graphs \cite{BL}.
The signless Laplacians of a simplicial complex \cite{Kaufman2} are systematically
studied with many interesting applications \cite{Kaufman,Lubotzky1,Lubotzky2,Lubotzky3}.
As a generalization, the (signless) $1$-Laplacian on simplicial complex was introduced by Luo and Zhang \cite{Luo} for the connection of its spectrum to the combinatorial properties.

In this paper, we study the largest eigenvalue of the $i$-th up Laplacian of a simplicial complex $K$, and show that it is not greater than the largest eigenvalue of the $i$-th up signless Laplacian of $K$; moreover, if $K$ is $(i+1)$-path connected, then the equality holds if and only if the $i$-th incidence signed graph $B_i(K)$ of $K$ is balanced.
As an application, we provide an upper bound for the largest eigenvalue of the $i$-th up Laplacian of a simplicial complex $K$, which improves the bound given by Horak and Jost \cite{HJ13B}, and generalizes the result of Anderson and Morley \cite{AM} from graphs to complexes.
The balancedness of $B_i(K)$ is an important property in discussion of the largest Laplacian eigenvalue.
We characterize the balancedness of the incidence graphs of simplicial complexes under wedge sum, join, Cartesian product, duplication of motifs, and construct an infinite family of $(i+1)$-path connected simplicial complexes $K$ with balanced $B_i(K)$ for each $i \ge 0$.

\section{Preliminaries}
\subsection{Simplical complex and Laplace operator}
Let $V$ be a finite set.
An \emph{abstract simplicial complex} (simply called a \emph{complex}) $K$ over $V$ is a collection of the subsets of $V$ which is closed under inclusion.
An \emph{$i$-face} or an \emph{$i$-simplex} of $K$ is an element of $K$ with cardinality $i+1$.
The \emph{dimension} of an $i$-face is $i$, and the dimension of $K$ is the maximum dimension of all faces of $K$.
The faces which are maximum under inclusion are called \emph{facets}.
We say $K$ is \emph{pure} if all facets have the same dimension.
So, a complex can  be considered as a hypergraph with facets as the edges of the hypergraph, and a pure complex will correspond to a uniform hypergraph, where a hypergraph is called \emph{uniform} if all edges have the same size.

We assume that $\emptyset \in K$, called the empty simplex with dimension $-1$.
Let $S_i(K)$ be the set of all $i$-faces of $\K$, where $S_{-1}(K)=\{\emptyset\}$.
The $p$-skeleton of $K$, written $K^{(p)}$, is the set of all simplices of $K$ of dimension less than or equal to $p$.
So, $K^{(1)}\backslash \{\emptyset\}$ is the usual graph with vertex set $V(K)$ consisting of $0$-faces usually called \emph{vertices}, and  edge set $E(K)$ consisting of $1$-faces usually called \emph{edges}.
We say $K$ is \emph{connected} if the graph $K^{(1)}\backslash \{\emptyset\}$ is connected.
An \emph{$i$-path} of length $m$ in $K$ is an ordering of $i$-simplices $F_1<F_2<\cdots<F_m$, such that $F_i \cap F_j$ is an $(i-1)$-face of $K$ if and only if $|j-l|=1$.
When $F_m$ coincides with $F_1$, we say that $L$ is an \emph{$i$-cycle} of length $(m-1)$.
The complex $K$ is \emph{$i$-path connected} if any two $i$-faces $F_1$ and $F_2$ of $K$ are connected by an $i$-path.

We say a face $F$ is \emph{oriented} if we assign an ordering of its vertices and write it as $[F]$.
Two ordering of the vertices of $F$ are said to determine the \emph{same orientation} if there is an even permutation transforming one ordering into the other.
If the permutation is odd, then the orientation are opposite.
The \emph{$i$-chain group} of $K$ over $\mathbb{R}$, denoted by $C_i(K,\mathbb{R})$, is the vector space over $\R$ generated by all oriented $i$-faces of $K$ modulo the relation $[F_1]+[F_2]=0$, where $[F_1]$ and $[F_2]$ are two different orientations of a same face.
The \emph{cochain group} $C^i(K,\R)$ is defined to be the dual of $C_i(K,\mathbb{R})$, i.e. $C^i(K,\R)=\text{Hom}(C_i(K,\mathbb{R}),\R)$, which are generated by the dual basis consisting of $[F]^*$ for all $F \in S_i(K)$, where
$$ [F]^*([F])=1, [F]^*([F'])=0 \text{~for~} F' \ne F.$$
The functions $[F]^*$ are called the \emph{elementary cochains}.
Note that $C_{-1}(K, \R)=\R \emptyset$, identified with $\R$,
and $C^{-1}(K, \R)=\R \emptyset^*$, also can be identified with $\R$, where $\emptyset^*$ is the identify function on the empty simplex.

For each integer $i=0,1,\ldots,\dim K$, The \emph{boundary map} $\p_i: C_i(K,\R) \to C_{i-1}(K,\R)$ is defined to be
$$\p_i([v_0,\ldots,v_i])=\sum_{j=0}^i (-1)^j[v_0,\ldots,\hat{v}_j,\ldots,v_i],$$
for each oriented $i$-face $[v_0,\ldots,v_i]$ of $K$,
where $\hat{v}_j$ denotes the vertex $v_j$ has been omitted.
In particular, $\p_0[v]=\emptyset$ for each $v \in S_0(K)$.
We will have the \emph{augmented chain complex} of $K$:
$$  \cdots \longrightarrow C_{i+1}(K,\R) \stackrel{\p_{i+1} }{\longrightarrow} C_i(K,\R) \stackrel{\p_{i} }{\longrightarrow} C_{i-1}(K,\R) \longrightarrow \cdots \longrightarrow  C_{-1}(K,\R) \longrightarrow 0,$$
satisfying $\p_i \circ \p_{i+1}=0$.
The \emph{$i$-th reduced homology group} of $K$ is defined to be $\tilde{H}_i(K)=\ker \p_i / \im \p_{i+1}$, and the dimension $\tilde{\beta}_i$ of $\tilde{H}_i(K)$ is called the \emph{$i$-th Betti number}.
The complex $K$ is called \emph{acyclic} if its reduced homology group vanishes in all dimensions.

Here, by abuse of notation, we use $\p \bar{F}$ to denote the set of all $i$-faces in the boundary of $\bar{F}$ when $\bar{F} \in S_{i+1}(K)$.
If $[\bar{F}]:=[v_0,\ldots,v_i]$ and $[F_j]:=[v_0,\ldots,\hat{v}_j,\ldots,v_i]$,
then we define $\sgn([F_j], \p[\bar{F}])=(-1)^j$, namely, the sign of $[F_j]$ appeared in $\p[\bar{F}]$, and $\sgn([F], \p[\bar{F}])=0$ if $F \notin \p \bar{F}$.
The \emph{coboundary map} $\delta_{i-1}: C^{i-1}(K,\R) \to C^i(K,\R)$ is the conjugate of $\p_i$ such that $ \d_{i-1} f = f \p_i.$
So
$$ (\d_{i-1} f)([v_0,\ldots,v_i])=\sum_{j=0}^i (-1)^jf([v_0,\ldots,\hat{v}_j,\ldots,v_i]).$$
Similarly, we have the \emph{augmented cochain complex} of $K$:
$$  \cdots \longleftarrow C^{i+1}(K,\R) \stackrel{\delta_{i} }{\longleftarrow} C^i(K,\R) \stackrel{\delta_{i-1} }{\longleftarrow} C^{i-1}(K,\R) \longleftarrow \cdots  {\longleftarrow} C^{-1}(K,\R) \longleftarrow  0,$$
satisfying $\delta_{i}\circ \delta_{i-1}=0$.
The \emph{$i$-th reduced cohomology group}  is defined to be $$\tilde{H}^i(K,\R)=\ker \delta_i / \im \delta_{i-1}.$$
As vector spaces, $\tilde{H}^i(K,\R)$ is the dual of $\tilde{H}_i(K,\R)$ and is isomorphic to $\tilde{H}_i(K,\R)$.

 By introducing inner products in $C^i(K,\R)$ and $C^{i+1}(K,\R)$ respectively, we have the adjoint $\d^*_i: C^{i+1}(K,\R) \to C^i(K,\R)$ of $\d_i$, which is defined by
$$ ( \delta_i f_1, f_2 )_{C^{i+1}} =( f_1, \delta^*_i f_2)_{C^{i}}$$ for all $f_1 \in C^i(K,\R), f_2 \in C^{i+1}(K,\R)$.

\begin{defi}\cite{HJ13B}
The following three operators are defined on $C^i(K,\R)$, where $K$ is a complex with an orientation $\s$.

(1) The $i$-dimensional combinatorial up Laplace operator or simply the \emph{$i$-up Laplace operator}: $$  \L_i^{up}(K,\s):=\d_i^* \d_i.$$

(2) The $i$-dimensional combinatorial down Laplace operator or the \emph{$i$-down Laplace operator}:
$$ \L_i^{down}(K,\s):=\d_{i-1}\d_{i-1}^*.$$

(3) The $i$-dimensional combinatorial Laplace operator or the \emph{$i$-Laplace operator}:
 $$\L_i(K,\s)=\L_i^{up}(K,\s)+\L_i^{down}(K,\s)=\delta^*_i \delta_i+\delta_{i-1}\delta^*_{i-1}.$$
 \end{defi}

Some we simply use $\L_i^{up}(K)$, $\L_i^{down}(K)$ and $\L_i(K)$ if the orientation $\s$ is clear  from the context.
In fact, the spectra of $\L_i^{up}(K)$ , $\L_i^{down}(K)$ and $\L_i(K)$ are all independent of the orientation; see Lemma \ref{reverse}.
All the there Laplacians are self-adjoint, nonnegative and compact.
Define a weight function on all faces of $K$:
$$ w: \bigcup_{i=-1}^{\dim K} S_i(K) \to \R^+,$$
so that the inner product in $C^i$ is defined as
$$ (f,g)_{C^i}=\sum_{F \in S_i(K)}f([F])g([F])([F]^*,[F]^*)=\sum_{F \in S_i(K)} w(F) f([F])g([F]).$$

In this paper, the weight or the inner product is implicit from the context for the Laplace operator.
If $w \equiv 1$ on all faces, then the underlying Laplacian is the \emph{combinatorial Laplace operator}, denoted by $L_i(K)$, as discussed in \cite{Duv, Fried}.
If the weights of all facets are equal to $1$, and $w$ satisfies the normalizing condition:
$$ w(F)=\sum_{\bar{F} \in S_{i+1}(K): F \in \p \bar{F}}w(\bar{F}),$$
for every $F \in S_i(K)$ which is not a facet of $K$,
then $w$ determines the \emph{normalized Laplace operator}, denoted by $\Delta_i(K)$, as analyzed in \cite{HJ13B}.

 Horak and Jost \cite{HJ13A, HJ13B} give explicit formulas for $\L_i^{up}$ and $\L_i^{down}$.
 Here we consider the matrix forms of the above operators.
 Let $D_i$ be the matrix of $\d_i: C^i \to C^{i+1}$ under the basis consisting of elementary cochains.
Then $D_i$ satisfies
 $$(D_i)_{[\bar{F}]^*, [F]^*}=\sgn([F], \p [\bar{F}]),$$
and the matrix $D_i^*$ of $ \d_i^*$ satisfies
$$ (D_i^*)_{[F]^*, [\bar{F}]^*}=\frac{w(\bar{F})}{w(F)} \sgn([F], \p [\bar{F}]).$$
So
$$ D_i^*=W_i^{-1} D_i^\top W_{i+1},$$
where $W_i$ and $W_{i+1}$ are diagonal matrices such that $(W_i)_{[F]^*, [F]^*}=w(F)$ and $(W_{i+1})_{ [\bar{F}]^*, [\bar{F}]^*}=w(\bar{F})$.
Hence
\begin{equation}\label{DefL} \L_i^{up}(K)=W_i^{-1} D_i^\top W_{i+1}D_i,~ \L_i^{down}(K)=D_{i-1} W_{i-1}^{-1}D_{i-1}^\top W_i.
\end{equation}

\subsection{Incidence-signed complex}
Let $K$ be a complex, and let $F, \bar{F} \in K$.
If $F \in \p \bar{F}$, then $(F,\bar{F})$ is an \emph{incidence} of $K$.
We will introduce the sign of incidences of $K$ and incidence-signed complex.

\begin{defi}
The \emph{incidence-signed complex} is a pair $(K, \varsigma)$, where $K$ is a complex, and $\varsigma: K \times K \to \{-1,0,1\}$ such that
$\varsigma(F, \bar{F}) \in \{-1,1\}$ if $F \in \p \bar{F}$, and $\varsigma(F, \bar{F})=0$ otherwise.
\end{defi}

The \emph{signed boundary map} $\p_i^\varsigma: C_i(K,\R) \to C_{i-1}(K,\R)$ is defined to be
$$\p_i^\varsigma[v_0,\ldots,v_i]=\sum_{j=0}^i (-1)^j[v_0,\ldots,\hat{v}_j,\ldots,v_i]
\varsigma(\{v_0,\ldots,\hat{v}_j,\ldots,v_i\},\{v_0,\ldots,v_i\}).$$

The \emph{signed co-boundary map} $\d^\varsigma_{i}: C^{i}(K,\R) \to C^{i+1}(K,\R)$ is the conjugate of $\p_{i+1}^\varsigma$, namely, for all $f \in C^{i}(K,\R)$,
$$\d^\varsigma_{i}f = f \p_{i+1}^\varsigma.$$

The \emph{signed adjoint}  $(\d_{i}^\varsigma)^*: C^{i+1}(K,\R) \to C^{i}(K,\R)$ is the adjoint of $\d_{i}^\varsigma$ satisfying
$$ ( \delta_{i}^\varsigma f_1, f_2 )_{C^{i+1}} =( f_1, (\delta^\varsigma_{i})^* f_2)_{C^{i}}$$ for all $f_1 \in C^{i}(K,\R), f_2 \in C^{i+1}(K,\R)$.

\begin{defi}
Let $(K,\varsigma)$ be a signed complex.

(1) The $i$-up Laplace operator of $(K,\varsigma)$ is defined to be  $\L_i^{up}(K,\varsigma)=(\d_i^\varsigma)^* \d_i^\varsigma$.

(2) The $i$-down Laplace operator of $(K,\varsigma)$ is defined to be $\L_i^{down}(K,\varsigma)=\d_{i-1}^\varsigma(\d_{i-1}^\varsigma)^*$.

(3) The $i$-Laplace operator of $(K,\varsigma)$ is defined to be $\L(K,\varsigma)=\L_i^{up}(K,\varsigma)+\L_i^{down}(K,\varsigma)$.

\end{defi}

Let $D_i^\varsigma$ be the matrix of $\d_i^\varsigma: C^i \to C^{i+1}$ under the basis consisting of elementary cochains.
Then
 $$(D_i^\varsigma)_{[\bar{F}]^*, [F]^*}=\sgn([F], \p [\bar{F}])\varsigma(F,\bar{F}).$$
The matrix $(D_i^\varsigma)^*$ of $ (\d_i^\varsigma)^*$ satisfies
$$ ((D_i^\varsigma)^*)_{[F]^*, [\bar{F}]^*}=\frac{w(\bar{F})}{w(F)} \sgn([F], \p [\bar{F}])\varsigma(F,\bar{F}),$$
where $w$ is a weight function on the faces of $K$.
So
$$ (D_i^\varsigma)^*=W_i^{-1} (D_i^\varsigma)^\top W_{i+1}.$$
Hence the matrices of $\L_i^{up}(K,\varsigma)$ and $\L_i^{down}(K,\varsigma)$ are respectively
\begin{equation}\label{LiupS} \L_i^{up}(K,\varsigma)=W_i^{-1} (D_i^\varsigma)^\top W_{i+1}D_i^\varsigma, \L_i^{down}(K,\varsigma)=D^\varsigma_{i-1} W_{i-1}^{-1}(D^\varsigma_{i-1})^\top W_i.
\end{equation}

If $\varsigma \equiv 1$, then
$ \L_i^{up}(K,\varsigma)=\L_i^{up}(K)$ and $\L_i^{down}(K,\varsigma)=\L_i^{down}(K).$
From this viewpoint, the incidence-signed complex $(K,\varsigma)$ is a generalization of the usual complex.
If $\varsigma(F,\bar{F}) = \sgn([F], \p [\bar{F}])$ for all pairs $(F,\bar{F})$ with $F \in \p \bar{F}$,
then  $D_i^\varsigma = |D_i|$, which is a nonnegative matrix, where $|A|:=[|a_{ij}|]$ if $A=[a_{ij}]$.
In this case, $\L_i^{up}(K,\varsigma)= |\L_i^{up}(K)|$, and
$\L_i^{down}(K,\varsigma)= |\L_i^{down}(K)|$.
The \emph{$i$-up signless Laplace operator} of $K$ is defined to be
$$ \Q_i^{up}(K):=|\L_i^{up}(K)|=W_i^{-1} |D_i|^\top W_{i+1}|D_i|,$$
and  the \emph{$i$-down signless Laplace operator} of $K$ is defined to be
$$ \Q_i^{down}(K):=|\L_i^{down}(K)|=|D_{i-1}|W_{i-1}^{-1} |D_{i-1}|^\top W_{i}.$$
Note that, if taking $W_0$ and $W_1$ both be identity matrix, then $\Q_0^{up}(K)$ is called the \emph{signless Laplacian} of the graph $K^{(1)}$ \cite{Hae}, which is studied in \cite{Des} for the nonbipartiteness of a graph; see the survey \cite{Cve}.

The incidence-signed complexes play an important role in the spectra of $2$-fold covering complexes.
Fan and Song \cite{Fan} proved that if $\phi:K \to L$ is a $2$-fold covering map from a complex $K$ to a complex $L$, and the weight satisfies $w_K(F)=w_L(\phi(F))$ for all $F\in S_i(K) \cup S_{i+1}(K)$, then
$$ \spec \L_i^{up}(K) = \spec \L_i^{up}(L) \cup \spec \L_i^{up}(L,\varsigma),$$
where $\varsigma$ is a sign on $(F,\bar{F}) \times S_i(L) \cup S_{i+1}(L)$ arisen from the covering.

\section{The largest eigenvalue of Laplace operator of simplicial complex}
In this section we will investigate the largest eigenvalue $\lamax(\L_i^{up}(K))$ of the $i$-up Laplaican $\L_i^{up}(K)$ of a complex $K$, where $\lamax(A)$ denotes the largest eigenvalue of a square matrix $A$ with all eigenvalues real.
By definition, $\lamax(\L_{i+1}^{down}(K))=\lamax(\L_i^{up}(K))$.
So it is enough to deal with the $i$-up Laplaican, and the results for the $(i+1)$-down Laplacian can be obtained similarly.

We should introduce some notions on signed graphs for preparation.
A \emph{signed graph} is a graph $G$ with a signing on its edges $\varsigma: E(G) \to \{1,-1\}$, denoted by $\Gamma=(G,\varsigma)$.
In particular, if $\varsigma \equiv 1$, namely all edges have positive signs, we will use $(G,+)$ to denote the signed graph.
The \emph{sign} of a subgraph $H$ of $G$, denoted by $\varsigma(H)$, is defined to be the product of the signs of all edges of $H$.
The signed graph $\Gamma$ is called \emph{balanced} if all signs of the cycles of $G$ are positive.
The signed graph and its balancedness were introduced by Harary \cite{Harary} on research of social psychology.

Switching is an important operation on signed graphs which preserves the balancedness.
A \emph{switching} $s_v$ on $\Gamma$ at a vertex $v$  is an operation which reverses the signs of all edges incident to $v$, and keeps the signs of other edges invariant.
The signed graphs $(G,\varsigma)$ is called \emph{switching equivalent} to a signed graph $(G,\varsigma')$ if $(G,\varsigma')$ can be obtained from $(G,\varsigma)$ by a sequence of switchings on vertices.
It is easy to see switching equivalence is an equivalent relation on signed graphs.

\begin{lem}\cite[Corollary 3.3]{Z1}\label{balance}
A signed graph $(G,\varsigma)$ is balanced if and only if it can be switched to $(G,+)$.
\end{lem}

Let $K$ be a complex with an orientation $\s$, and let $F, \bar{F} \in K$.
If $F \in \p \bar{F}$, then $(F,\bar{F})$ is an incidence of $K$.
The \emph{$i$-th incidence graph} $B_i(K)$ is a bipartite graph with vertex set $S_i(K)\cup S_{i+1}(K)$ such that $\{F,\bar{F}\}$ is an edge if and only if $F\in\partial\bar{F}$.
Actually $B_i(K)$ is a signed graph such that the sign of an edge $\{F,\bar{F}\}$ is given by  $\sgn([F],\partial[\bar{F}])$, denoted by $(B_i(K),\s)$.
Note that a \emph{signature matrix} is a diagonal matrix with $\pm 1$ on its diagonal.

\begin{lem}\label{reverse}
Let $K$ be a complex with an orientation $\s$, and let $(B_i(K),\s)$ be the corresponding $i$-th incidence signed graph.
Let $\tau$ be another orientation of $K$.
Then the following results hold.

$(1)$ $(B_i(K),\tau)$ is obtained from $(B_i(K),\s)$ only by applying
 a switching $s_F$ on $B_i(K)$ on an $i$- or $(i+1)$-face $F$ of $K$
 if and only if $\tau$ is obtained from $\s$ only by reversing
the orientation of the face $F$, where the reversed orientation of a $0$-face $[v]$ is defined to be $-[v]$.

$(2)$ If $\tau$ is obtained from $\s$ only by reversing
the orientation of $i$-face $F$, then
$$ \L_i^{up}(K,\tau)=S_F \L_i^{up}(K,\s)S_F, \L_i^{down}(K,\tau)=S_F \L_i^{down}(K,\s)S_F$$
where $S_F$ is a signature matrix defined on the $i$-th elementary cochains of $K$ with only $-1$ on the diagonal corresponding $[F]^*$.

$(3)$ If $\tau$ is obtained from $\s$ only by reversing
the orientation of an $(i+1)$-face $F$, then
$$ \L_i^{up}(K,\tau)=\L_i^{up}(K,\s);$$
and if $\tau$ is obtained from $\s$ only by reversing
the orientation of an $(i-1)$-face $F$,
$$ \L_i^{down}(K,\tau)=\L_i^{down}(K,\s).$$
\end{lem}

\begin{proof}
(1) For each $F \in S_i(K)$ and $\bar{F} \in S_{i+1}(K)$ such that $F \in \p \bar{F}$, as
$$ \sgn(-[F], \p [\bar{F}])=-\sgn([F], \p [\bar{F}])=\sgn([F], \p (- [\bar{F}])),$$
the result (1) follows.

(2)  Let $D_i,D_{i-1}$ be the matrices of $\d_i,\d_{i-1}$ with respect to the orientation $\s$, respectively.
Then
the matrices of $\d_i,\d_{i-1}$ with respect to the orientation $\tau$ are $D_i^{\tau}:=D_i S_F$ and $D_{i-1}^{\tau}:=S_F D_{i-1}$.
By (\ref{DefL}),
$ \L_i^{up}(K,\s)=W_i^{-1} D_i^\top W_{i+1}D_i$.
So
$$ \L_i^{up}(K,\tau)=W_i^{-1}(D_iS_F)^\top W_{i+1} D_iS_F=S_F (S_F W_i^{-1} S_F) D_i^\top W_{i+1} D_i S_F=S_F \L_i^{up}(K,\s)S_F,$$
as $S_F=S_F^\top=S_F^{-1}$ and $S_F W_i^{-1} S_F=W_i^{-1}$.
The proof of the second equality is similar.

(3) The matrices of $\d_i$ with respect to  the orientation $\tau$ is $D_i^{\tau}:=S_F D_i$, so that
$$ \L_i^{up}(K,\tau)=W_i^{-1} (S_F D_i)^\top W_{i+1} S_F D_i=W_i^{-1} D_F^\top (S_F W_{i+1} S_F) D_i=\L_i^{up}(K,\s).$$
The proof of the second equality is similar by noting that the matrices of $\d_{i-1}$ with respect to  the orientation $\tau$ is $D_{i-1}^{\tau}:= D_{i-1}S_F$.
\end{proof}

By Lemma \ref{reverse}, we find that the different orientations of faces results in the Laplace operators which are similar by a signature matrix.
So the Laplacian spectrum is independent of the orientations.

Observe that the complex $K$ is $(i+1)$-path connected if and only if the $i$-th incidence graph $B_i(K)$ is connected.
When discussing the spectrum of $\L_i^{up}(K)$, it suffices to deal with the case of $K$ being $(i+1)$-path connected; otherwise, $\L_i^{up}(K)$ is a direct sum of the $i$-up Laplacians of some $(i+1)$-path connected components.

\begin{thm}\label{balance}
Let $K$ be a complex with an orientation $\s$.
Then
\begin{equation}\label{main}
\lamax(\L_i^{up}(K))\leq\lamax(\Q_i^{up}(K)),
\end{equation}
with equality if $B_i(K)$ is balanced.
If further $K$ is $(i+1)$-path connected,
then the equality holds in (\ref{main}) if and only if $(B_i(K),\s)$ is balanced.
\end{thm}

\begin{proof}
Let $\lambda:=\lamax(\L_i^{up}(K))$ be the largest eigenvalue of $\L_i^{up}(K)$ with $f$ as an associated eigenfunction. Then we have
$$(\lambda f,f)_{C^i}=(\L_i^{up}(K)f,f)_{C^i}=(\delta_i^*\delta_i f,f)_{C^i}=(\delta_i f ,\delta_i f)_{C^i}=(f\partial_{i+1},f\partial_{i+1})_{C_i}.$$
So,
\begin{equation}\label{inner1}
\begin{aligned}
\lambda(f,f)_{C^i}&=\sum\limits_{\bar{F}\in
S_{i+1}(K)}((f\partial_{i+1})([\bar{F}]))^2\omega(\bar{F})\\
&=\sum\limits_{\bar{F}\in
S_{i+1}(K)}\left(f\left(\sum\limits_{F\in\partial\bar{F}}
\sgn([F],\partial[\bar{F}])[F]\right)\right)^2\omega(\bar{F})\\
&=\sum\limits_{\bar{F}\in
S_{i+1}(K)}\left(\sum\limits_{F\in\partial\bar{F}}
\sgn([F],\partial[\bar{F}])f([F])\right)^2\omega(\bar{F})\\
&\leq\sum\limits_{\bar{F}\in S_{i+1}(K)}\left(\sum\limits_{F\in\partial\bar{F}}
|f([F])|\right)^2\omega(\bar{F})\\
&= (\Q_i^{up}(K)|f|, |f|).
\end{aligned}
\end{equation}
As $(f,f)_{C^i}=(|f|,|f|)_{C^i}$, by the Min-Max Theorem \cite{HJ13B}, we have
$$ \lamax(\L_i^{up}(K)) =\la \le \lamax(\Q_i^{up}(K)).$$

Now suppose that $K$ is $(i+1)$-path connected and $ \lamax(\L_i^{up}(K)) = \lamax(\Q_i^{up}(K))$.
By Eq. (\ref{inner1}),
$|f|$ is an eigenfunction of $\Q_i^{up}(K)$ associated with $\lamax(\Q_i^{up}(K))$.
As $K$ is $(i+1)$-path connected, $\Q_i^{up}(K)$ is nonnegative and irreducible.
So, by Perron-Frobenious Theorem of nonnegative matrices, $|f|$ is a positive vector, which implies that $f([F]) \ne 0$ for all $F \in S_i(K)$.
By the last inequality in (\ref{inner1}),
for each $ \bar{F}\in S_{i+1}(K)$,
\begin{equation}\label{inner2}
\left(\sum\limits_{F\in\partial\bar{F}}
\sgn([F],\partial[\bar{F}])f([F])\right)^2=\left(\sum\limits_{F\in\partial\bar{F}}
|f([F])|\right)^2.
\end{equation}
Replacing $[F]$ by $-[F]$ if $f([F])<0$, then Eq. (\ref{inner2}) still holds as
$$\sgn(-[F],\partial([\bar{F}])f(-[F])=\sgn([F],\partial[\bar{F}])f([F]),
|f(-[F])|=|f([F])|.$$
So we can assume that $f$ is positive.
Similarly, replacing $[\bar{F}]$ by $-[\bar{F}]$,  Eq. (\ref{inner2}) also holds as
$$\sum\limits_{F\in\partial\bar{F}}
\sgn([F],\partial(-[\bar{F}])f([F])=-\sum\limits_{F\in\partial\bar{F}}
\sgn([F],\partial([\bar{F}])f([F]).$$
So we can assume for each $  \bar{F}\in S_{i+1}(K)$, there exists an $F \in \p  \bar{F}$ such that $\sgn([F],\partial([\bar{F}])=1$.
Now returning to Eq. (\ref{inner2}), by the above two assumptions, for all $\bar{F} \in S_{i+1}(K)$ and $F \in \p \bar{F}$,
$$ \sgn([F],\partial[\bar{F}])>0,$$
which implies that $(B_i(K), \s)$ can be transformed into $(B_i(K), +)$ by a sequence of switchings on $i$-faces and/or $(i+1)$-faces by
Lemma \ref{reverse}(1).
So, $(B_i(K), \s)$ is balanced by Lemma \ref{balance}.

Conversely, if $(B_i(K), \s)$ is balanced, then by Lemma \ref{balance}, $(B_i(K), \s)$ is switching equivalent to $(B_i(K), +)$ by applying a sequence of switchings at some $i$-faces and/or $(i+1)$-faces of $K$.
Let $\tau$ be the resulting orientation of $K$ after the above switchings.
By Lemma \ref{reverse}(2-3),
$$\L_i^{up}(K, \s)=S \Q_i^{up}(K)S,$$
where $S$ is the product of $S_F$ over the reversed $i$-faces of $K$.
So $\L_i^{up}(K, \s)$ is similar to $\Q_i^{up}(K)$, and hence have the same largest eigenvalues.
\end{proof}

The following result for combinatorial Laplacian of graphs is known.
Here we give a proof for the completeness and also for the general Laplacian with arbitrary weights.

\begin{cor}
Let $K$ be a connected complex.
Then
\begin{equation}\label{main1}
\lamax(\L_0^{up}(K))\leq\lamax(\Q_0^{up}(K)),
\end{equation}
with equality  if and only if the $1$-skeleton $K^{(1)}$ of $K$ is bipartite.
\end{cor}

\begin{proof}
Let $\s$ be an orientation of $K$.
By Theorem \ref{main}, it suffices to prove that $(B_0(K),\s)$ is balanced if and only if $K^{(1)}$ is bipartite.
Let $C$ be a cycle of $K^{(1)}$ on vertices $v_1,\ldots,v_k$.
Then $B_0(K)$ contains a cycle $v_1, \{v_1,v_2\}, v_2, \ldots, v_k, \{v_k,v_1\},v_1$, whose sign is
$$ \varsigma(C)=\prod_{i=1}^k \sgn([v_i], \p [e_i])\sgn([v_{i+1}], \p [e_i]),$$
where $[e_i]=[v_i,v_{i+1}]$ or $[e_i]=-[v_i,v_{i+1}]$, and $v_{k+1}=v_1$.
As each term in the above product is $-1$, $s(C)=(-1)^k$.
So, if  $(B_0(K),\s)$ is balanced, then $k$ is even, and hence $K^{(1)}$ is bipartite.
Conversely, if $K^{(1)}$ is bipartite, then $k$ is even so that $s(C)=1$, implying that $(B_0(K),\s)$ is balanced.
\end{proof}

As an application of Theorem \ref{main}, we give an upper bound for the largest Laplacian eigenvalue.
Note that the \emph{degree} $\deg F$ of an $i$-face $F$ of a complex $K$ is defined to be the sum of the weights of all simplices $\bar{F}$ such that $F \in \p \bar{F}$, namely,
$\deg F=\sum_{\bar{F} \in S_{i+1}(K): F \in \p \bar{F}} w(\bar{F})$.

\begin{thm}
Let $K$ be a complex.
Then
\begin{equation}\label{upp1}\lamax(\L_i^{up}(K)) \le \max_{F \in S_{i+1}(K)} \sum_{E \in \p F} \frac{\deg E}{w(E)}.
\end{equation}
In particular, for combinatorial Laplacian (or $w \equiv 1$),
\begin{equation}\label{upp2}\lamax(L_i^{up}(K)) \le \max_{F \in S_{i+1}(K)} \sum_{E \in \p F} \deg E;
\end{equation}
and if $K$ is $(k+1)$-path connected, then the equality in (\ref{upp2}) holds if and only if $B_i(K)$ is balanced and the sum of the degrees of all $i$-faces in each $(i+1)$-face of $K$ is constant.
\end{thm}

\begin{proof}
By Theorem \ref{main},
$$ \lamax(\L_i^{up}(K)) \le \lamax(\Q_i^{up}(K))=\lamax(\Q_{i+1}^{down}(K)).$$
For an $F \in S_{i+1}(K)$, let $r_{F}$ be the row sum of $\Q_{i+1}^{down}(K)$ corresponding to $[F]^*$.
We have
$$ r_{F}=\sum_{E \in \p F} \frac{w(F)}{w(E)} + \sum_{F'\in S_{i+1}(K): \atop F \cap F'=E \in S_i(K)} \frac{w(F')}{w(E)}=\sum_{E \in \p F} \sum_{F' \in S_{i+1}(K): \atop E \in \p F'}\frac{w(F')}{w(E)}=\sum_{E \in \p F} \frac{\deg E}{w(E)}.$$
So we have
\begin{equation}\label{pf-upp}
\lamax(\Q_{i+1}^{down}(K))  \le \max_{F \in S_{i+1}(K)} r_{F}
 = \max_{F \in S_{i+1}(K)} \sum_{E \in \p F} \frac{\deg E}{w(E)},
\end{equation}
which yields the upper bound in (\ref{upp1}).

If $w \equiv 1$, we have
$$\lamax(L_i^{up}(K)) \le \max_{F \in S_{i+1}(K)} \sum_{E \in \p F} \deg E,
$$
Note that $\Q_{i+1}^{down}(K)$ is nonnegative and irreducible as $K$ is $(k+1)$-path connected.
If the equality holds, then
$$\lamax(L_i^{up}(K)) =\lamax(Q_i^{up}(K))=\lamax(Q_{i+1}^{down}(K))= \max_{F \in S_{i+1}(K)} \sum_{E \in \p F} \deg E.$$
From the first equality, we know
$B_i(K)$ is balanced from Theorem \ref{main};
and from the third equality, we find that $r_F$ is constant as $Q_{i+1}^{down}(K)$ is nonnegative and irreducible, namely, the sum of the degrees of all $i$-faces in each $(i+1)$-face of $K$ is constant.
Conversely, if $r_F$ is constant, then $\lamax(Q_{i+1}^{down}(K))=r_F$.
Also, as $B_i(K)$ is balanced, by Theorem \ref{main}, $\lamax(L_i^{up}(K)) =\lamax(Q_i^{up}(K))=\lamax(Q_{i+1}^{down}(K))$.
The result follows.
\end{proof}

\begin{rmk}
Anderson and Morley \cite{AM} proved that for a graph $G$, the largest eigenvalue of the Laplacian matrix $L(G)$ (namely, $L_0^{up}(G)$) holds
$$ \lamax(L(G)) \le \max_{\{u,v\} \in E(G)} (\deg u + \deg v).$$
So our result (\ref{upp2}) generalizes the above bound from graphs to complexes.

Horak and Jost \cite{HJ13B} proved that
$$\lamax(\L_i^{up}(K)) \le (i+2) \frac{\max_{E \in S_i(K)} \deg E}{\min_{E \in S_i(K)} w(E)}.$$
For an $F \in S_{i+1}(K)$, it has $(i+2)$ $i$-faces $E \in \p F$, so
$$\sum_{E \in \p F} \frac{\deg E}{w(E)} \le  \frac{(i+2) \max_{E \in S_i(K)} \deg E}{\min_{E \in S_i(K)} w(E)}.$$
So our bound in (\ref{upp1}) improves the bound given by  Horak and Jost.
Duval and Reiner \cite{Duv} proved that
$$\lamax(L_i^{up}(K)) \le n,$$
where $n$ is the number of the vertices of $K$.
Our bound in (\ref{upp2}) will be better for large $n$ and small degrees of $i$-faces.
In particular, if $n \ge (i+2)\max_{E \in S_i(K)} \deg E ~(\ge \max_{F \in S_{i+1}(K)} \sum_{E \in \p F} \deg E)$, then our bound is better.
\end{rmk}

\section{Balancedness of incidence graphs of complexes under operations}

By Theorem \ref{main}, we know that the balancedness of the $i$-th incidence signed graph $B_i(K)$ of a complex $K$ plays an important role on the largest eigenvalue of the Laplacian of $K$.
In this section, we will study the balancedness of the incidence signed graphs of complexes
under some operations such as  wedge sum, join, Cartesian product and  duplication of motifs.

\subsection{Wedge}

\begin{defi}\cite{HJ13B}
Let $K_1$ and $K_2$ be complexes on disjoint vertex sets $V(K_1)$ and $V(K_2)$, respectively.
Let $F_1=\{v_0,\ldots, v_k\} \in S_k(K_1)$ and $F_2=\{u_0,\ldots, u_k\} \in S_k(K_2)$.
The combinational \emph{$k$-wedge sum} of $K_1$ and $K_2$ is a complex obtained from $K_1 \cup K_2$ by identifying  $F_1$ with $F_2$ such that $v_i$ is identified with $u_i$ for $i=0,1,\ldots,k$.
\end{defi}

The above definition generalizes in an obvious way to the $k$-wedge sum of arbitrary many complexes.
If $k \ge 1$, the $k$-wedge sum depends on the ways of identification of vertices. In fact, there exists a bijection $\phi: F_1 \to F_2$ such that $v \in F_1$ is identified with $\phi(v)$.

\begin{thm}\label{wedge}
Let $K=K_{1}\vee_k K_2$ be a $k$-wedge sum of complexes $K_1$ and $K_2$ by identifying a $k$-face $F_1$ of $K_1$ and a $k$-face $F_2$ of $K_2$.
If $i \ge k-1$, the $i$-th incidence graph $B_i(K)$ of $K$ is balanced if and only if $B_i(K_1)$ and $B_i(K_2)$ both are balanced.
\end{thm}

\begin{proof}
Csae 1: $i>k$. By definition, $B_i(K)$ is only involved with $i$-faces and $(i+1)$-faces of $K_1$ and $K_2$.
So, in this case, $B_i(K)$ is a disjoint union of $B_i(K_1)$ and $B_i(K_2)$,
which implies the result immediately.

Case 2: $i=k$. In this case, $B_i(K)$ is a $0$-wedge sum of $B_i(K_1)$ and $B_i(K_2)$ by identifying the vertex $F_1$ of $B_i(K_1)$ and the vertex $F_2$ of $B_i(K_2)$ and forming a new vertex $F$ (or a new $i$-face of $K$).
Then $F$ is a cut vertex of $B_i(K)$.
So, any cycle of $B_i(K)$ is either a cycle of $B_i(K_1)$ or a cycle of $B_i(K_2)$ but not both.
The result follows in this case.

Case 3: $i=k-1$. As $B_i(K_1)$ and $B_i(K_2)$ are the subgraphs of $B_i(K)$,
they are both balanced if $B_i(K)$ is balanced.
Now suppose $B_i(K_1)$ and $B_i(K_2)$ both are balanced.
Let $C$ be a cycle of $B_i(K)$.
If $C$ lies in $B_i(K_1)$ or $B_i(K_2)$, surely it is positive.
Otherwise, $C$ contains both vertices of  $B_i(K_1)$ and $B_i(K_2)$.

Subcase 3.1: $C$ contains the face $F$ (the identification of $F_1$ with $F_2$).
We may list the vertices of $C$ as follows:
\begin{equation}\label{wedge-cycle}
C: F, G_1, P_{12}, G_2, P_{23}, G_3, \cdots, G_{t-1}, P_{t-1,t}, G_t, F,
\end{equation}
where $G_j \in \p F$ for $j \in [t]$, and $P_{j,j+1}$ is a path connecting $G_j$ and $G_{j+1}$ which lies in $K_1$ if $j$ is odd and lies in $K_2$ otherwise.
Now inserting some $F$'s and $G_j$'s in the above sequence, we get a circuit $\tilde{C}$:
\begin{equation}
\tilde{C}: F, G_1, P_{12}, G_2, \textcolor{blue}{F}, \textcolor{blue}{G_2}, P_{2,3}, G_3, \textcolor{blue}{F}, \textcolor{blue}{G_3} \cdots, G_{t-1}, \textcolor{blue}{F}, \textcolor{blue}{G_{t-1}}, P_{t-1,t}, G_t, F.
\end{equation}
Let $C_j:=F, G_j, P_{j,j+1}, G_{j+1}, F$, a cycle in $B_i(K_1)$ or $B_i(K_2)$ depending on whether $j$ is odd or even.
Then $\tilde{C}$ is a union of $C_j$ for $j \in [t-1]$.
As each cycle in $K_1$ or $K_2$ is positive, the sign of $C$ holds
$$ \varsigma(C)=\varsigma(\tilde{C})=\varsigma(C_1)\cdots \varsigma(C_{t-1})=1,$$
which implies the result.

Subcase 3.2:  $C$ contains no the face $F$.
The discussion is similar as Subcase 3.1.
The vertices of $C$ are listed as follows:
\begin{equation}
C: G_1, P_{12}, G_2, P_{2,3}, G_3, \cdots, G_{t-1}, P_{t-1,t}, G_t,
\end{equation}
where $G_t=G_1$, and $G_j$ and $P_{j,j+1}$ are defined as in (\ref{wedge-cycle}).
We get  a circuit $\tilde{C}$:
\begin{equation}
\tilde{C}: \textcolor{blue}{F}, G_1, P_{12}, G_2, \textcolor{blue}{F}, \textcolor{blue}{G_2}, P_{2,3}, G_3, \textcolor{blue}{F}, \textcolor{blue}{G_3}, \cdots, G_{t-1}, \textcolor{blue}{F}, \textcolor{blue}{G_{t-1}}, P_{t-1,t}, G_t, \textcolor{blue}{F}.
\end{equation}
Let $C_j:=F, G_j, P_{j,j+1}, G_{j+1}, F$, a cycle in $B_i(K_1)$ or $B_i(K_2)$ depending on whether $j$ is odd or even.
Then $\tilde{C}$ is a union of $C_j$ for $j \in [t-1]$, and
$$ \varsigma(C)=\varsigma(\tilde{C})=\varsigma(C_1)\cdots \varsigma(C_{t-1})=1,$$
which also implies the result.
\end{proof}

\begin{rmk}\label{rmk-wedge}
Let $K=K_{1}\vee_k K_2$ with the identified $k$-face $F$.
If $i \le k-2$, then $B_i(K)$ is always unbalanced.
 As $B_i(F)$ is a subgraph of $B_i(K)$, it suffices to show $B_i(F)$ is unbalanced.
As $B_i(F)$ is only related with $i$- and $(i+1)$-faces of $F$,
it is enough to show $B_i(F')$ is unbalanced for an $(i+2)$-face $F'$ of $F$,
which is proved in the following lemma.
\end{rmk}

\begin{lem}\label{balance-simp}
Let $F$ be a $k$-simplex for $k \ge 2$.
Then the incidence signed graph $B_{k-2}(F)$ is unbalanced.
\end{lem}

\begin{proof}
Let $F=\{i_0,\ldots,i_k\}$, and let $F_{i_j}=F \backslash \{i_j\}$, $F_{i_ji_l}=F \backslash \{i_j,i_l\}$ for $j,l =0,1,\ldots,k$ and $j \ne l$.
We order the vertices of $F$ as $i_0<i_1<\cdots<i_k$, and all faces of $F$ are oriented according the ordering of vertices.

If $k$ is odd, then $B_{k-2}(F)$ contains the following cycle $C$ listed in Fig. \ref{wedge1}, with the signs:
$$ s([F_{i_ji_{j+1}}],[F_{i_j}])=s([F_{i_ji_{j+1}}],[F_{i_{j+1}}])=(-1)^j, j \in [k-1],$$
and
$$ s([F_{i_1i_{k}}],[F_{i_k}])=-1, s([F_{i_1i_{k}}],[F_{i_1}])=(-1)^{k-1}=1.$$
So, the sign of $C$ is negative, implying that $B_{k-2}(F)$ is unbalanced.

\begin{figure}[h]
	\centering
			\begin{tikzpicture}
				\tikzstyle{vertex}=[circle,fill=black,minimum size=3pt];
				
				\node[vertex] (v1) at (-6,1) {};
				\node[vertex] (v2) at (-4,1) {};
				\node[vertex] (v3) at (-2,1) {};
                \node[vertex] (v4) at (1,1) {};
				\node[vertex] (v5) at (3,1) {};
				\node[vertex] (v6) at (-6,2.6) {};
				\node[vertex] (v7) at (-4,2.6) {};
				\node[vertex] (v8) at (-2,2.6) {};
                \node[vertex] (v9) at (1,2.6) {};
				\node[vertex] (v10) at (3,2.6) {};
				
				\node at (-6,0.5){$[F_{i_1i_2}]$};
                \node at (-4,0.5){$[F_{i_2i_3}]$};
                \node at (-2,0.5){$[F_{i_3i_4}]$};
                \node at (-0.5,1){$\cdots$};
                \node at (1,0.5){$[F_{i_{k-1}i_k}]$};
				\node at (3,0.5){$[F_{i_1i_k}]$};
                \node at (-6,3.1){$[F_{i_1}]$};
                \node at (-4,3.1){$[F_{i_2}]$};
                \node at (-2,3.1){$[F_{i_3}]$};
                \node at (-0.5,2.6){$\cdots$};
                \node at (1,3.1){$[F_{i_{k-1}}]$};
				\node at (3,3.1){$[F_{i_k}]$};

				\draw (v1) -- (v6);
                \draw (v1) -- (v7);
                \draw (v2) -- (v7);
                \draw (v2) -- (v8);
                \draw (v3) -- (v8);
				\draw (v3) -- (v8);
                \draw (v4) -- (v9);
				\draw (v4) -- (v10);
                \draw (v5) -- (v10);
                \draw (v5) -- (v6);
			\end{tikzpicture}
			\caption{A cycle $C$ of $B_{k-2}(F)$ for $k$ being odd}\label{wedge1}
	\end{figure}
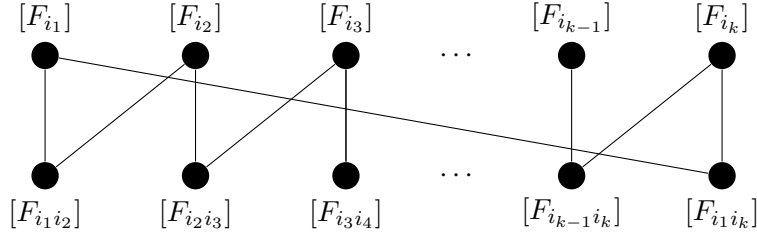

If $k$ is even, then $B_{k-2}(F)$ contains the following cycle $C$ listed in Fig. \ref{wedge2}, with the signs:
$$ s([F_{i_ji_{j+1}}],[F_{i_j}])=s([F_{i_ji_{j+1}}],[F_{i_{j+1}}])=(-1)^j, j =0,1, \ldots,k-1$$
and
$$ s([F_{i_0i_{k}}],[F_{i_k}])=1, s([F_{i_0i_{k}}],[F_{i_0}])=(-1)^{k-1}=-1.$$
So, the sign of $C$ is negative, implying that $B_{k-2}(F)$ is also unbalanced.

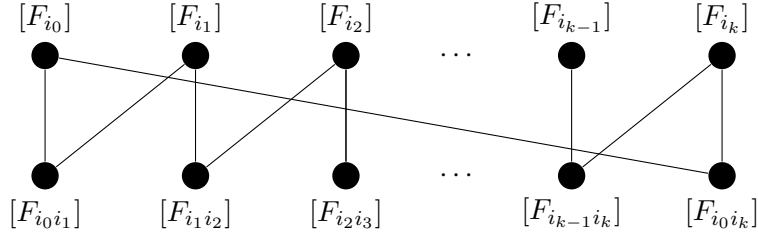
\begin{figure}[h]
	\centering
			\begin{tikzpicture}
				\tikzstyle{vertex}=[circle,fill=black,minimum size=3pt];
				
				\node[vertex] (v1) at (-6,1) {};
				\node[vertex] (v2) at (-4,1) {};
				\node[vertex] (v3) at (-2,1) {};
                \node[vertex] (v4) at (1,1) {};
				\node[vertex] (v5) at (3,1) {};
				\node[vertex] (v6) at (-6,2.6) {};
				\node[vertex] (v7) at (-4,2.6) {};
				\node[vertex] (v8) at (-2,2.6) {};
                \node[vertex] (v9) at (1,2.6) {};
				\node[vertex] (v10) at (3,2.6) {};
				
				\node at (-6,0.5){$[F_{i_0i_1}]$};
                \node at (-4,0.5){$[F_{i_1i_2}]$};
                \node at (-2,0.5){$[F_{i_2i_3}]$};
                \node at (-0.5,1){$\cdots$};
                \node at (1,0.5){$[F_{i_{k-1}i_k}]$};
				\node at (3,0.5){$[F_{i_0i_k}]$};
                \node at (-6,3.1){$[F_{i_0}]$};
                \node at (-4,3.1){$[F_{i_1}]$};
                \node at (-2,3.1){$[F_{i_2}]$};
                \node at (-0.5,2.6){$\cdots$};
                \node at (1,3.1){$[F_{i_{k-1}}]$};
				\node at (3,3.1){$[F_{i_k}]$};

				\draw (v1) -- (v6);
                \draw (v1) -- (v7);
                \draw (v2) -- (v7);
                \draw (v2) -- (v8);
                \draw (v3) -- (v8);
				\draw (v3) -- (v8);
                \draw (v4) -- (v9);
				\draw (v4) -- (v10);
                \draw (v5) -- (v10);
                \draw (v5) -- (v6);
			\end{tikzpicture}
			\caption{A cycle $C$ of $B_{k-2}(F)$ for $k$ being even}\label{wedge2}
	\end{figure}
\end{proof}

\begin{cor}\label{infi-wedge}
For each integer $i \ge 0$, there exist infinitely many $(i+1)$-path connected acyclic complexes $K$ with $B_i(K)$ being balanced.
\end{cor}

\begin{proof}
For each integer $i \ge 0$, let $\sigma$ be an $(i+1)$-simplex.
Let $K_0=\sigma$, and $K_p=K_{p-1} \vee_i K_0$ for $p \ge 1$.
We will get a sequence of complexes $\{K_i\}_{i \ge 0}$.
From the Case 2 in the proof of Theorem \ref{wedge},
$B_i(K_p)$ is a $0$-wedge sum of $B_i(K_{p-1})$ and $B_i(K_0)$.
It is known that $B_i(K_0)$ is connected and hence $B_i(K_p)$ is connected by induction, which implies that $K_p$ is $(i+1)$-path connected.
Also, $B_i(K_0)$ contains no cycles and hence $B_i(K_p)$ contains no cycles by induction, which implies that $B_i(K_p)$ is balanced.
Finally noting that $K_{p-1} \cap K_0$ is an $i$-simplex which is acyclic,
by Mayer-Vietoris sequence (see \cite{Munk}) we have
$\tilde{H}_q(K_p) \cong \tilde{H}_q(K_{p-1}) \oplus \tilde{H}_q(K_0)$.
As $K_0$ is acyclic, we get $K_p$ is acyclic by induction.
\end{proof}

In Corollary \ref{infi-wedge}, for $i=0$, taking $\s$ be an edge, we will get a sequence of trees; for $i=1$, taking $\s$ be an triangle, we will get a sequence of acyclic planar $2$-complexes.

\subsection{Join}
\begin{defi}\cite{HJ13B}
Let $K_1$ and $K_2$ be complexes on disjoint vertex sets $V(K_1)$ and $V(K_2)$ respectively. The \emph{join} $K_1\ast K_2$ is a complex on the vertex set $V(K_1) \cup V(K_2)$, whose faces are $F_1\ast F_2:=F_1 \cup F_2$, where $F_1$ is a face in $K_1$ and $F_2$ is a face in $K_2$.
\end{defi}

\begin{thm}
Let $K=K_1\ast K_2$ be the join of complexes $K_1$ and $K_2$.
If $\dim K_1 + \dim K_2 \ge i+1$,
then the $i$-th incidence graph $B_i(K)$ of $K$ is unbalanced.
\end{thm}

\begin{proof}
By definition, $K$ has an $(i+2)$-face $F$.
By Lemma \ref{balance-simp}, $B_i(F)$ is unbalanced so that $B_i(K)$ is unbalanced.
\end{proof}

\begin{rmk}
Let $K=K_1\ast K_2$ be the join of complexes $K_1$ and $K_2$.
If $\dim K_1 + \dim K_2 \le i$, then
neither $K_1$ nor $K_2$ contains an $(i+1)$-face, then $B_i(K_1)$ and $B_i(K_2)$ are both empty graphs (without edges).
In particular, if $\dim K_1 + \dim K_2 < i$, then $K$ contains no $(i+1)$-faces, implying that $B_i(K)$ is also empty.
If $\dim K_1 + \dim K_2 = i$, then $B_i(K)$ contains $(i+1)$-faces.
We will not leave much space for a discussion of the balancedness of $B_i(K)$  in this case.
But it is worth to mention that if $K_1,K_2$ are both $0$-dimensional complexes (empty graphs), then $K$ is a complete bipartite graph for which the incidence graph $B_0(K)$ is balanced.
\end{rmk}


\subsection{Cartesian product}
\begin{defi}
Let $K_1$ and $K_2$ be complexes with vertex sets $V(K_1)$ and $V(K_2)$ respectively.
The \emph{Cartesian product} $K_1\Box  K_2$ of $K_1$ and $K_2$ is a complex with  vertex set $V(K_1)\times V(K_2)$, whose faces are $F\times v:=F\times \{v\}=\{(u_0,v),\ldots,(u_s,v)\}$ if  $F=\{u_0,\ldots,u_s\}\in K_1$ and $v\in V(K_2)$, or $u\times F':=\{u\}\times F'=\{(u,v_0),\ldots,(u,v_t)\}$ if $u\in V(K_1)$ and $F'=\{v_0,\ldots,v_t\}\in K_2$.
\end{defi}

An orientation of $F$ in $K_1$ induces an orientation of $F\times v$ in $K_1 \Box K_2$.
If $[F]=[u_0,\ldots,u_s]\in K_1$, then
$[F\times v]:=[(u_0,v),\ldots,(u_s,v)]$.
So
$\sgn([F\times v], \partial[\bar{F}\times v])=\sgn([F],\partial[\bar{F}]).$
Similarly, an orientation of $F'$ in $K_2$ induces an orientation of $u\times F'$.
Given weight functions $w_1$ on $K_1$ and $w_2$ on $K_2$.
We define the weight $w$ on $K_1\Box K_2$ such that for any $F \in K_1$ and $F' \in K_2$,
\begin{equation}\label{Cart-weight}
w(F\times v):=w_1(F), ~ w(u\times F'):=w_2(F').
\end{equation}
This implies a condition on $w_1$ and $w_2$, namely, $w(u,v)=w_1(u)=w_2(v)$ for any vertex $u$ of $K_1$ and $v$ of $K_2$.

Note that the following Lemma \ref{laplace} for the (combinatorial) Laplacian matrix of a graph was observed by Fiedler \cite[Item 3.4]{Fied}.
Here we give a proof for general weight function, and use $\text{Spec} A$ to denote the spectrum of a square matrix $A$.

\begin{lem}\label{laplace}
Let $K_1\Box K_2$ be the Cartesian product of complexes $K_1$ and $K_2$ with weight defined in (\ref{Cart-weight}). Then
\begin{equation}\label{Cart1}\L_0^{up}(K_1\Box K_2)=\L_0^{up}(K_1)\otimes I_{V(K_2)}+I_{V(K_1)}\otimes \L_0^{up}(K_2),
\end{equation}
and
\begin{equation}\label{spec}
\text{Spec}\L_0^{up}(K_1\Box K_2)=\{\lambda+\mu: \lambda\in \text{Spec}\L_0^{up}(K_1),\mu\in \text{Spec}\L_0^{up}(K_2)\},
\end{equation}
where $I_{V(K_1)},I_{V(K_2)}$ are the identity matrices defined on $V(K_1),V(K_2)$ respectively.
\end{lem}

\begin{proof}
First we discuss the Eq. (\ref{Cart1}).
The diagonal entry of
$\L_0^{up}(K_1\Box K_2)$ corresponding to a vertex $(u,v)$ is:
\begin{align*}
\L_0^{up}(K_1\times K_2)_{(u,v),(u,v)}&=\sum\limits_{\bar{F}\in S_1(K_1\times K_2): (u,v)\in \partial \bar{F}}\frac{w(\bar{F})}{w(u,v)}\\
&=\sum\limits_{F\in S_1(K_1): (u,v)\in F\times v}\frac{w(F\times v)}{w(u,v)}+\sum\limits_{F'\in S_1(K_2): (u,v)\in u\times F'}\frac{w(u\times F')}{w(u,v)}\\
&=\sum\limits_{F\in S_1(K_1): u\in \partial F}\frac{w_1(F)}{w_1(u)}+\sum\limits_{F'\in S_1(K_2): v\in \partial F'}\frac{w_2(F')}{w_2(v)}\\
&=(L_0^{up}(K_1)\otimes I_{V(K_2)})_{(u,v),(u,v)}+(I_{V(K_1)}\otimes L_0^{up}(K_2))_{(u,v),(u,v)}.
\end{align*}

Next we consider the off-diagonal entries of $\L_0^{up}(K_1\Box K_2)$. Take two vertices $(u_1,v_1)$ and $(u_2,v_2)$ of $K_1\Box K_2$.
 If $u_1\neq u_2$ and $v_1\neq v_2$, then
  $\L_0^{up}(K_1\times K_2)_{(u_1,v_1),(u_2,v_2)}=0$ as no $1$-faces $\bar{F}$ of $K_1\Box K_2$ contain both $(u_1,v_1)$ and $(u_2,v_2)$.
  Also the $((u_1,v_1),(u_2,v_2))$-entry of right side of Eq. (\ref{Cart1}) is zero as $(I_{V(K_2)})_{v_1,v_2}=0$ and $(I_{V(K_1)})_{u_1,u_2}=0$.
So it suffices to consider the cases: (1) $u_1=u_2:=u, v_1\neq v_2$, (2) $u_1\neq u_2, v_1=v_2:=v$.

For the case (1), we have
\begin{align*}
\L_0^{up}(K_1\Box K_2)_{(u,v_1),(u,v_2)}&=\sum_{F' \in S_1(K_2): \atop v_1,v_2 \in \p F'}\frac{w(u \times F')}{w(u,v_1)}\sgn ((u,v_1),\partial[u\times F']) \sgn ((u,v_2),\partial[u\times F'])\\
&=\sum_{F' \in S_1(K_2): \atop v_1,v_2 \in \p F'}\frac{w_2(F')}{w_2(v_1)}\sgn (v_1,\partial[F'])\cdot \sgn (v_2,\partial[F'])\\
&=\L_0^{up}(K_2)_{v_1,v_2}\\
&=\left(L_0^{up}(K_1)\otimes I_{V(K_2)}+I_{V(K_1)}\otimes L_0^{up}(K_2)\right)_{(u,v_1),(u,v_2)}.
\end{align*}
Similarly, for the case (2), we have
\begin{align*}
\L_0^{up}(K_1\Box K_2)_{(u_1,v),(u_2,v)}&=\sum_{F \in S_1(K_1): \atop u_1,u_2 \in \p F}\frac{w(F \times v)}{w(u_1,v)}
\sgn ((u_1,v),\partial[F\times v])\cdot \sgn ((u_2,v),\partial[F\times v])\\
&=\sum_{F \in S_1(K_1): \atop u_1,u_2 \in \p F}\frac{w_1(F)}{w_1(u_1)}\sgn (u_1,\partial[F])\cdot \sgn (u_2,\partial[F])\\
&=\L_0^{up}(K_1)_{u_1,u_2}\\
&=\left(L_0^{up}(K_1)\otimes I_{V(K_2)}+I_{V(K_1)}\otimes L_0^{up}(K_2)\right)_{(u_1,v),(u_2,v)}.
\end{align*}
So we proved the Eq. (\ref{Cart1}).

Let $f$ be an eigenvector of $\L_0^{up}(K_1)$ associated with an eigenvalue $\la$, and $g$ be an eigenvector of $\L_0^{up}(K_2)$ associated with an eigenvalue $\mu$.
Then $\L_0^{up}(K_1)f=\lambda f$ and $\L_0^{up}(K_2)g=\mu g$, and by Eq. (\ref{Cart1})
\begin{align*}
\L_0^{up}(K_1\Box K_2)(f\otimes g)&=(\L_0^{up}(K_1)\otimes I_{V(K_2)})(f\otimes g)+(I_{V(K_1)}\otimes \L_0^{up}(K_2))(f\otimes g)\\
&=(\L_0^{up}(K_1)f)\otimes(I_{V(K_2)}g)+(I_{V(K_1)}f)\otimes(\L_0^{up}(K_2)g)\\
&=(\lambda+\mu)(f\otimes g).
\end{align*}
So, $f\otimes g$ is an eigenvector of $\L_0^{up}(K_1\Box K_2)$ associated with the eigenvalue $\la+\mu$.
As $\L_0^{up}(K_1)$ (respectively, $\L_0^{up}(K_2)$) is self-adjoint, it has a set of orthonomal eigenvectors consisting of the basis of $\R^{V(K_1)}$ (respectively, $\R^{V(K_2)}$).
So, $\L_0^{up}(K_1\Box K_2)$ has a set of orthonomal eigenvectors consisting of the basis of $\R^{V(K_1)\times V(K_2)}$.
The Eq. (\ref{spec}) follows.
\end{proof}

\begin{rmk}\label{signless}
By a similar proof as in Lemma \ref{laplace}, we have the following result:
\begin{equation}\label{Cart1-sl}\Q_0^{up}(K_1\Box K_2)=\Q_0^{up}(K_1)\otimes I_{V(K_2)}+I_{V(K_1)}\otimes \Q_0^{up}(K_2),
\end{equation}
and
\begin{equation}\label{spec-sl}
\text{Spec}\Q_0^{up}(K_1\Box K_2)=\{\lambda+\mu: \lambda\in \text{Spec}\Q_0^{up}(K_1),\mu\in \text{Spec}\Q_0^{up}(K_2)\}.
\end{equation}
\end{rmk}

\begin{lem}\label{Cart-i}
Let $K_1\Box K_2$ be the Cartesian product of complexes $K_1$ and $K_2$.
Then for $i \ge 1$, $B_i(K_1 \Box K_2)$ is the union of following disjoint graphs:
$|V(K_2)|$ copies of $B_i(K_1)$, and $|V(K_1)|$ copies of $B_i(K_2)$.
Consequently,
\begin{equation}\label{Cart2}
\L_i^{up}(K_1\Box K_2)=\L_i^{up}(K_1)\otimes I_{V(K_2)} \oplus I_{V(K_1)}\otimes \L_i^{up}(K_2),
\end{equation}
where $I_{V(K_1)},I_{V(K_2)}$ are the identity matrices defined on $V(K_1),V(K_2)$ respectively.
\end{lem}

\begin{proof}
Let $v$ be a vertex of $K_2$.
Taking $\{v\}$ as a complex consisting of the only vertex $v$, we have the Cartesian product $K_1 \Box \{v\}$, which is isomorphic to $K_1$.
So, $B_i(K_1 \Box \{v\})$ is isomorphic to $B_i(K_1)$, and hence
$B_i(K_1 \Box K_2)$ contains $|V(K_2)|$ disjoint copies of $B_i(K_1)$, namely, $B_i(K_1 \Box \{v\})$ for $v \in V(K_2)$.
Similarly, $B_i(K_1 \Box K_2)$ contains $|V(K_1)|$ disjoint copies of $B_i(K_2)$, namely $B_i(\{u\} \Box K_2)$ for $u \in V(K_1)$.

Observe that $B_i(K_1 \Box v)$ and $B_i(u \Box K_2)$ share no common vertices as $F \times v \ne u \times F'$ for any $F \in S_i(K_1) \cup S_{i+1}(K_1)$ and $F' \in S_i(K_2) \cup S_{i+1}(K_2)$.
So, $B_i(K_1 \Box K_2)$ contains the union of following disjoint graphs:
$|V(K_2)|$ copies of $B_i(K_1)$ and $|V(K_1)|$ copies of $B_i(K_2)$.
Also, there exists no edges between any two ones of the above graphs, as
$F \times v \notin \p (\bar{F}\times v') \cup \p (u \times \bar{F'})$ and
$u \times F' \notin \p(u' \times \bar{F'}) \cup  \p(\bar{F} \times v)$ for any
$F \in S_i(K_1)$, $F'\in S_i(K_2)$, $\bar{F} \in S_{i+1}(K_1)$,  $\bar{F'} \in S_{i+1}(K_2)$, $u \ne u'$ and $v \ne v'$.
So we prove the first result.

By the above discussion, $\L_i^{up}(K_1\Box K_2)$ is a direct sum of  $\L_i^{up}(K_1\Box \{v\})$ for $v \in V(K_2)$ and $\L_i^{up}(\{u\}\Box K_2)$ for $u \in V(K_1)$, which implies that Eq. (\ref{Cart2}) by noting that $w(F \times v)=w_1(v)$ and $w(u\times F')=w_2(F')$.
\end{proof}

Note that $(K_1\Box K_2)^{(1)}=K_1^{(1)} \Box K_2^{(1)}$, and $K_1^{(1)} \Box K_2^{(1)}$ is connected if and only if $K_1^{(1)}$ and $K_1^{(2)}$ are both connected (as graphs) \cite{IK}.

\begin{thm}
Let $K_1, K_2$ be $(i+1)$-path connected complexes.
Then $B_i(K_1\Box K_2)$ is balanced if and only if $B_i(K_1)$ and $B_i(K_2)$ are both balanced.
\end{thm}

\begin{proof}
Firstly, we consider the case of $i=0$.
As $K_1, K_2$ be $1$-path connected, $K_1^{(1)}$ and $K_1^{(2)}$ are both connected.
So, $(K_1\Box K_2)^{(1)}$ is connected and hence $K_1\Box K_2$ is $1$-path connected.
Suppose that $B_0(K_1 \Box K_2)$ is balanced.
Then by Lemma \ref{balance} and Remark \ref{signless},
\begin{equation}\label{0-bal-1}\lamax(\L_0^{up}(K_1\Box K_2))= \lamax(\Q_0^{up}(K_1\Box K_2))=\lamax(\Q_0^{up}(K_1))+\lamax(\Q_0^{up}(K_2)).
\end{equation}
So, also by Lemma \ref{balance},
\begin{equation}\label{0-bal-2}
\lamax(\L_0^{up}(K_1))=\lamax(\Q_0^{up}(K_1)), \lamax(\L_0^{up}(K_2))=\lamax(\Q_0^{up}(K_2)),
\end{equation}
which implies that $B_0(K_1)$ and $B_0(K_2)$ are both balanced.
Conversely, if $B_0(K_1)$ and $B_0(K_2)$ are both balanced, then Eq. (\ref{0-bal-2}) holds and hence Eq. (\ref{0-bal-1}) holds.
So $B_0(K_1 \Box K_2)$ is balanced.

For $i \ge 1$, by Lemma \ref{Cart-i}, $B_i(K_1 \Box K_2)$ is the union of following disjoint graphs:
$|V(K_2)|$ copies of $B_i(K_1)$, and $|V(K_1)|$ copies of $B_i(K_2)$.
So the result follows immediately in this case.
\end{proof}

\subsection{Duplication of motifs}
\begin{defi}\cite{HJ13B}
Let $K$ be a complex and $S$ be a collection of simplices in $K$.
The \emph{closure} of $S$ written $\cl S$ is the smallest subcomplex of $K$ that contains all simplices in $S$.
The \emph{star} of $S$ written $\st S$ is the set of all simplices in $K$ that have a face in $S$.
The \emph{link} of $S$, written $\lk S$,  is defined to be $\cl \st S-\st \cl S$.
\end{defi}

\begin{defi}\cite{HJ13B}
Let $K$ be a complex. The subcomplex $\Sigma$ of $K$ is a \emph{motif} if $\Sigma$ contains all simplices in $K$ on the vertices of $V(\Sigma)$.
\end{defi}

\begin{defi}\cite{HJ13B}\label{motif}
Let $\Sigma$ be a subcomplex of a complex $K$. $\Sigma$ is called a \emph{$i$-motif } if the following conditions hold.

(1) ($\forall F_1,F_2\in \Sigma$), if $F_1,F_2\subset F \in K$, then $F\in \Sigma$.

(2) dim $\lk\Sigma=i$.
\end{defi}

\begin{lem}\cite{HJ13B}
If $K$ is an $(i+1)$-path connected complex, then any motify satisfying (1) in Definition \ref{motif} will have a link of dimension $i$.
\end{lem}

\begin{defi}
Tow complexes $K$ and $L$ are isomorphic if there exists a bijection $f: V(K)\rightarrow V(L)$ such that $\{v_0,\ldots, v_k\}\in K$ if and only if $\{f(v_0),\ldots, f(v_k)\}\in L$.
\end{defi}

Now let $\Sigma$ be an $i$-motif of $K$ with
vertex set $V(\Sigma)=\{v_0,\ldots, v_k\}$.
Let $V(\lk\Sigma)=\{u_0,\ldots, u_m\}$.
 According to the definition of link, $V(\Sigma)\cap V(\lk\Sigma)=\emptyset$.
 Define a complex $\Sigma'$ with vertex set $V(\Sigma')=\{v'_0,\ldots, v'_k\}$,
 which is isomorphic to $\Sigma$ by the isomorphism  $f: v_i \mapsto v'_i$.
 Then
$$K^{\Sigma}:=K\cup \{\{v'_{i_0},\ldots, v'_{i_l},u_{j_1},\ldots, u_{j_s}\}\mid\{v_{i_0},\ldots, v_{i_l},u_{j_1},\ldots, u_{j_s}\}\in K\}$$
 is a complex obtained from $K$ by the \emph{duplication of the $i$-motif $\Sigma$}.
 Denote by $K_{\Sigma'}=(K \backslash \st \Sigma) \cup \st \Sigma'$, a subcomplex of $K^\Sigma$ which is isomorphic to $K$ via a map $\bar{f}$ with $\bar{f}|_{V(\Sigma)}=f$ and  $\bar{f}|_{V(K)\backslash V(\Sigma)}=id$.

\begin{thm}\label{Motif-balance}
Let $K$ be an $(i+1)$-path  connected complex and let $\Sigma$ be  an $i$-motif of $K$.
Then $B_i(K^{\Sigma})$ is balanced if and only if $B_i(K)$ is balanced.
\end{thm}

\begin{proof}
As $K$ is a subcomplex of $K^{\Sigma}$,  $B_i(K)$ is a subgraph of $B_i(K^{\Sigma})$.
So, if  $B_i(K^{\Sigma})$ is balanced, then $B_i(K)$ is balanced.

Conversely, suppose that $B_i(K)$ is balanced.
Then $B_i(K_{\Sigma'})$ is also balanced as $K_{\Sigma'}$ is isomorphic to $K$.
Let $C$ be any cycle of $B_i(K^{\Sigma})$.
If $C$ is contained in $B_i(K)$ or $B_i(K_{\Sigma'})$, then $C$ is positive as $B_i(K)$ and $B_i(K_{\Sigma'})$ are both balanced.
 Otherwise, $C$ contains both faces $\bar{F}\in \st\Sigma$ and $\bar{F'}\in \st\Sigma'$. Note that there is an isomorphism $\bar{f}$ from $K$ to $K_{\Sigma'}$.
So, $\bar{f}$ induces a bijection between faces in $\st\Sigma$ and the faces in $\st\Sigma^{'}$, together with their orientations, namely, if $[w_0,\ldots,w_t]$ is an oriented face of $\st\Sigma$, then $[\bar{f}(w_0),\ldots,\bar{f}(w_t)]$ is an oriented face of $\st\Sigma'$.
 Now replace each $F'\in S_i(\st\Sigma')\cup S_{i+1}(\st \Sigma')$ appeared in $C$ by $\bar{f}^{-1}(F')\in S_i(\st\Sigma)\cup S_{i+1}(\st \Sigma)$.
 We get a circuit $\tilde{C}$ in $B_i(K)$, which is a union of edge-disjoint cycles in $B_i(K)$, say $C_1,\ldots,C_t$.
 Then
$$\varsigma(C)=\varsigma(\tilde{C})=\varsigma(C_1)\times\cdots\times \varsigma(C_t)=1,$$
as each $C_i$ in $B_i(K)$ is positive.
So $B_i(K^{\Sigma})$ is balanced.
\end{proof}

By Theorem \ref{Motif-balance}, one can apply the duplication of motifs to construct an infinite family of $(i+1)$-path connected complexes with $B_i(K)$ being balanced for each $i \ge 0$.

\end{document}